\documentclass[12pt,a4paper]{amsart}
\usepackage{graphicx}
\usepackage{comment}
\usepackage{tikz-cd}
\usetikzlibrary{decorations.pathmorphing}
\usetikzlibrary{calc}
\usepackage{amsmath,amssymb,enumerate,amsfonts,graphicx}

\usepackage{algorithm,algorithmic}

\usepackage{color}
\usepackage{epsfig}
\usepackage[numbers,sort&compress]{natbib}
\usepackage[T1]{fontenc}
\usepackage{algorithm,algorithmic}
\usepackage{stmaryrd}
\usepackage{paralist}
\usepackage{geometry}
\usepackage{tikz}
\usepackage{todonotes}
\tikzset{black node/.style={draw, circle, fill = black, minimum size = 5pt, inner sep = 0pt}}
\tikzset{normal/.style = {draw=none, fill = none, minimum size =0, rectangle}}

\newcommand{\R}{\mathbb{R}}

\newcommand{\Z}{\mathbb{Z}}
\newcommand{\Q}{\mathbb{Q}}

\DeclareMathOperator{\OPT}{\mathrm{OPT}}
\DeclareMathOperator{\LP}{\mathrm{LP}}
\DeclareMathOperator{\SA}{\mathsf{SA}}

\newcommand{\eps}{\varepsilon}

\newcommand{\cVD}{\textsc{Cluster-VD}}
\newcommand{\cVDapx}{\textsc{Cluster-VD-apx}}
\newcommand{\VC}{\textsc{Vertex Cover}}
\newcommand{\FVST}{\textsc{FVST}}

\usepackage{aliascnt}
\newtheorem{theorem}{Theorem}
\newaliascnt{prop}{theorem}
\newtheorem{prop}[prop]{Proposition}
\aliascntresetthe{prop}

\newaliascnt{lemma}{theorem}
\newtheorem{lemma}[lemma]{Lemma}
\aliascntresetthe{lemma}

\newaliascnt{observation}{theorem}

\aliascntresetthe{observation}

\newaliascnt{corollary}{theorem}

\aliascntresetthe{corollary}

\newaliascnt{conjecture}{theorem}

\aliascntresetthe{conjecture}

\newaliascnt{claim}{theorem}

\aliascntresetthe{claim}

\theoremstyle{definition}

\begin{document}
\title[A Tight Approximation Algorithm for CVD]{A Tight Approximation Algorithm for the\\ Cluster Vertex Deletion Problem}
%
%


%
\author[M.~Aprile]{Manuel Aprile}
\author[M.~Drescher]{Matthew Drescher}
\author[S.~Fiorini]{Samuel Fiorini}
\author[T.~Huynh]{Tony Huynh}
\address[M.~Aprile, M.~Drescher, S.~Fiorini]{\newline D\'epartement de Math\'ematique
\newline Universit\'e libre de Bruxelles
\newline Brussels, Belgium}
\email{manuelf.aprile@gmail.com, knavely@gmail.com, sfiorini@ulb.ac.be}
\address[T.~Huynh]{\newline School of Mathematics
\newline Monash University
\newline Melbourne, Australia}
\email{tony.bourbaki@gmail.com}

\thanks{This project was supported by ERC Consolidator Grant 615640-ForEFront. Samuel Fiorini and Manuel Aprile are also supported by FNRS grant T008720F-35293308-BD-OCP. Tony Huynh is also supported by the Australian Research Council.}

\date{\today}
\sloppy

%
\begin{abstract}
We give the first $2$-approximation algorithm for the cluster vertex deletion problem. This is tight, since approximating the problem within any constant factor smaller than $2$ is UGC-hard. Our algorithm combines the previous approaches, based on the local ratio technique and the management of true twins, with a novel construction of a ``good'' cost function on the vertices at distance at most $2$ from any vertex of the input graph.

As an additional contribution, we also study cluster vertex deletion from the polyhedral perspective, where we prove almost matching upper and lower bounds on how well linear programming relaxations can approximate the problem.  

\keywords{Approximation algorithm \and Cluster vertex deletion \and Linear programming relaxation \and Sherali-Adams hierarchy.}
\end{abstract}
\maketitle 

\section{Introduction} \label{sec:intro}

A \emph{cluster graph} is a graph that is a disjoint union of complete graphs. 
Let $G$ be any graph. A set $X \subseteq V(G)$ is called a \emph{hitting set} if $G - X$ is a cluster graph. Given a graph $G$ and (vertex) cost function $c : V(G) \to \Q_{\geqslant 0}$, the \emph{cluster vertex deletion} problem (\cVD) asks to find a hitting set $X$ whose cost $c(X) := \sum_{v \in X} c(v)$ is minimum. We denote by $\OPT(G,c)$ the minimum cost of a hitting set.

If $G$ and $H$ are two graphs, we say that $G$ \emph{contains} $H$ if some \emph{induced} subgraph of $G$ is isomorphic to $H$. Otherwise, $G$ is said to be \emph{$H$-free}. Denoting by $P_k$ the path on $k$ vertices, we easily see that a graph is a cluster graph if and only if it is $P_3$-free. Hence, $X \subseteq V(G)$ is a hitting set if and only if $X$ contains a vertex from each induced $P_3$.

\cVD{} has applications in graph modeled data clustering in which an unknown set of samples may be contaminated. An optimal solution for \cVD{} can recover a clustered data model, retaining as much of the original data as possible~\cite{HKMN10}. Vertex deletion problems such as \cVD{}, where one seeks to locate vertices whose removal leaves a graph with desirable properties, often arise when measuring robustness and attack tolerance of real-life networks \cite{albert2000error, aprile2019graph, jahanpour2013analysis}.

From what precedes, \cVD{} is a hitting set problem in a $3$-uniform hypergraph, and as such has a ``textbook'' $3$-approximation algorithm (see for instance the introduction of \cite{fiorini2016improved}). 
Moreover, the problem has an approximation-preserving reduction from \VC. By adding a pendant edge to each vertex of $G$, one checks that solving \cVD{} on the new graph is equivalent to solving \VC{} on the original graph (see Proposition \ref{prop:EFlowerbound} for more details). Hence, obtaining a $(2-\eps)$-approximation algorithm for some $\eps > 0$ would contradict either the Unique Games Conjecture or P $\neq$ NP. 

The first non-trivial approximation algorithm for \cVD{} was a $5/2$-approximation due to You, Wang and Cao~\cite{YWC17}. Shortly afterward, Fiorini, Joret and Schaudt gave a $7/3$-approximation~\cite{fiorini2016improved}, and subsequently a $9/4$-approximation~\cite{FJS20}. 

\subsection{Our contribution}\label{sec:ourcontribution}

In this paper, we close the gap between $2$ and $9/4 = 2.25$ and prove the following \emph{tight} result.

\begin{theorem}\label{thm:2apx}
\cVD{} has a $2$-approximation algorithm.
\end{theorem}

All previous approximation algorithms for \cVD{} are based on the local ratio technique. See the survey of Bar-Yehuda, Bendel, Freund, and Rawitz \cite{bbfr2004} for background on this standard algorithmic technique. Our algorithm is no exception, see Algorithm~\ref{algo} below. 
However, it significantly differs from previous algorithms in its crucial step (see Step~\ref{step:find_good_H} below). In fact, almost all our efforts in this paper focus on that particular step of the algorithm (see Theorem~\ref{thm:localratio}), which searches for a special type of induced subgraph  of $G$, which we now describe.

Let $H$ be an induced subgraph of $G$, and let $c_H: V(H) \to \Q_{\geqslant 0}$. The weighted graph $(H,c_H)$ is said to be \emph{$\alpha$-good in $G$} (for some factor $\alpha \geqslant 1$) if $c_H$ is not identically $0$ and $c_H(X \cap V(H)) \leqslant \alpha \cdot \OPT(H, c_H)$ holds for every (inclusionwise) minimal hitting set $X$ of $G$. We overload terminology and say that an induced subgraph $H$ is \emph{$\alpha$-good in $G$} if there exists a cost function $c_{H}$ such that $(H,c_{H})$ is $\alpha$-good in $G$. We stress that the local cost function $c_H$ is defined obliviously of the global cost function $c : V(G) \to \Q_{\geqslant 0}$.

A pair of vertices $u, u'$ of $G$ are called \emph{twins}\footnote{We warn the reader that, in other papers, \emph{twins} are usually called \emph{true twins}, whereas two vertices which have the same set of neighbours are called \emph{false twins}.  Since we have no need of false twins in this paper, we have chosen to use \emph{twins} in place of \emph{true twins}.} if $uu' \in E(G)$, and for all $v \in V(G-u-u')$, $uv \in E(G)$ if and only if $u'v \in E(G)$.  We say that $G$ is \emph{twin-free} if $G$ has no twins.  As in~\cite{fiorini2016improved,FJS20}, if $G$ has a pair of twins $u,u'$, then \cVD{} admits an easy reduction step (see Steps~\ref{step:true_twins_start}--\ref{step:true_twins_end}).  The idea is simply to add the cost of $u'$ to that of $u$ and delete $u'$. This works since $u'$ belongs to a minimal hitting set of $G$ if and only if $u$ does (see \cite{fiorini2016improved} for a complete proof).  Therefore, when searching for $\alpha$-good induced subgraphs, we may assume that $G$ is twin-free, which is crucial for our proofs.

\begin{algorithm}\caption{$\cVDapx(G,c)$}\label{algo} 
\begin{algorithmic}[1]
\REQUIRE $(G,c)$ a weighted graph
\ENSURE $X$ a minimal hitting set of $G$
\IF{$G$ is a cluster graph}
	\STATE{$X \leftarrow \varnothing$}
\ELSIF{there exists $u \in V(G)$ with $c(u) = 0$}
	\STATE{$G' \leftarrow G - u$}
	\STATE{$c'(v) \leftarrow c(v)$ for $v \in V(G')$}
	\STATE{$X' \leftarrow \cVDapx(G',c')$ \label{step:zero_cost}}
	\STATE{$X \leftarrow X'$ if $X'$ is a hitting set of $G$; $X \leftarrow X' \cup \{u\}$ otherwise}
\ELSIF{there exist twins $u, u' \in V(G)$ \label{step:true_twins_start}}
	\STATE{$G' \leftarrow G - u'$}
	\STATE{$c'(u) \leftarrow c(u) + c(u')$; $c'(v) \leftarrow c(v)$ for $v \in V(G'-u)$}
	\STATE{$X' \leftarrow \cVDapx(G',c')$ \label{step:true_twins}}
	\STATE{$X \leftarrow X'$ if $X'$ does not contain $u$; $X \leftarrow X' \cup \{u'\}$ otherwise \label{step:true_twins_end}}
\ELSE
	\STATE{find a weighted induced subgraph $(H,c_H)$ that is $2$-good in $G$ \label{step:find_good_H}}
	\STATE{$\lambda^* \leftarrow \max \{\lambda \mid \forall v \in V(H) : c(v) - \lambda c_H(v) \geqslant 0\}$} \label{step:lambda_star}
	\STATE{$G' \leftarrow G$}
	\STATE{$c'(v) \leftarrow c(v) - \lambda^* c_H(v)$ for $v \in V(H)$; $c'(v) \leftarrow c(v)$ for $v \in V(G) \setminus V(H)$}\label{step:newcost}
	\STATE{$X \leftarrow \cVDapx(G',c')$ \label{step:increase_dual}}
\ENDIF
\STATE{return $X$}
\end{algorithmic}
\end{algorithm}

We will use two methods to establish $\alpha$-goodness of induced subgraphs. We say that $(H,c_H)$ is \emph{strongly} $\alpha$-good if $c_H$ is not identically $0$ and $c_H(V(H)) \leqslant \alpha \cdot \OPT(H,c_H)$. Clearly, if  $(H,c_H)$ is strongly $\alpha$-good then $(H, c_H)$ is $\alpha$-good in $G$, for every graph $G$ which contains $H$.  
We say that $H$ itself is \emph{strongly $\alpha$-good} if $(H,c_H)$ is strongly $\alpha$-good for some cost function $c_H$.

Let $N_{\leqslant i}[v]$ (resp.\ $N_i(v)$) be the set of vertices at distance at most (resp.\ equal to) $i$ from vertex $v$.  We abbreviate $N(v):=N_1(v)$ and $N[v]:=N_{\leqslant 1}[v]$.  
If we cannot find a strongly $\alpha$-good induced subgraph in $G$, we will find an induced subgraph $H$ that has a special vertex $v_0$ such that $N[v_0]$ is entirely contained in $H$, and a cost function $c_H : V(H) \to \Z_{\geqslant 0}$ such that $c_H(v) \geqslant 1$ for all vertices $v \in N[v_0]$ and $c_H(V(H)) \leqslant \alpha \cdot \OPT(H,c_H) + 1$. Notice that no minimal hitting set $X$ can contain all the vertices of $N[v_0]$, since if $X$ contains $N(v_0)$, then $v_0$ is an isolated clique. Hence, $c_H(X \cap V(H)) \leqslant c_H(V(H)) - 1 \leqslant \alpha \cdot \OPT(H,c_H)$ and so $(H,c_H)$ is $\alpha$-good in $G$. We say that $(H,c_H)$ (sometimes simply $H$) is \emph{centrally} $\alpha$-good (in $G$) \emph{with respect to $v_0$}. Moreover, we call $v_0$ the \emph{root vertex}.

In order to illustrate these ideas, consider the following two examples (see Figure \ref{fig:C4_P3}). First, let $H$ be a $C_4$ (that is, a $4$-cycle) contained in $G$ and $\mathbf{1}_H$ denote the unit cost function on $V(H)$. Then $(H,\mathbf{1}_H)$ is strongly $2$-good, since $\sum_{v \in V(H)} \mathbf{1}_H(v) = 4 = 2 \OPT(H,\mathbf{1}_H)$. Second, let $H$ be a $P_3$ contained in $G$, starting at a vertex $v_0$ that has degree-$1$ in $G$. Then $(H,\mathbf{1}_H)$ is centrally $2$-good with respect to $v_0$, but it is not strongly $2$-good.

\begin{figure}[t]\centering 
 \begin{minipage}[t]{.42\textwidth}\centering
    \begin{tikzpicture}[scale=.3,inner sep=2pt]
\tikzstyle{vtx}=[circle,draw,thick,fill=white]
\draw[thick] (0,0)--(4,0)--(4,4)--(0,4)--cycle;
\draw (0,0) node[vtx]{\tiny $1$}
	(4,0) node[vtx]{\tiny $1$}
	(4,4) node[vtx]{\tiny $1$}
	(0,4) node[vtx]{\tiny $1$};
\end{tikzpicture}
\label{fig:C4}
\end{minipage}
\qquad
\begin{minipage}[t]{.42\textwidth}\centering
 \begin{tikzpicture}[scale=.3,inner sep=2pt]
\tikzstyle{vtx}=[circle,draw,thick,fill=white]
\draw[thick] (-4,0)--(0,0)--(4,0);
\draw (0,0) node[vtx]{\tiny $1$}
	(-4,0) node[vtx,fill=black!25]{\tiny $1$}
	(4,0) node[vtx]{\tiny $1$};
\draw[thick] (4,0) ellipse (6 and 4);
\end{tikzpicture}
 \end{minipage}  
\caption{
$(C_4,\mathbf{1}_{C_4})$ on the left is strongly $2$-good. $(P_3,\mathbf{1}_{P_3})$, on the right, is centrally $2$-good in $G$ with respect to the gray vertex, which has degree 1 in $G$.} \label{fig:C4_P3}
\end{figure}
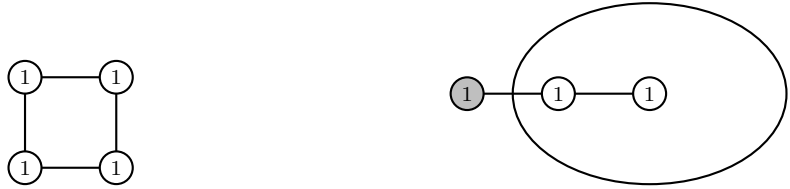

Each time we find a $2$-good weighted induced subgraph $H$ in $G$, the local ratio technique allows us to recurse on an induced subgraph $G'$ of $G$ in which at least one vertex of $H$ is deleted from $G$. For example, the $2$-good induced subgraphs mentioned above allow us to reduce to input graphs $G$ that are $C_4$-free and have minimum degree at least $2$.

The crux of our algorithm, Step~\ref{step:find_good_H},
relies on the following structural result. 

\begin{theorem} \label{thm:localratio} 
Let $G$ be a twin-free graph, let $v_0$ be any vertex of $G$, and let $H$ be the subgraph of $G$ induced by $N_{\leqslant 2}[v_0]$. There exists a cost function $c_H : V(H) \to \Z_{\geqslant 0}$ such that $(H,c_H)$ is either strongly $2$-good, or centrally $2$-good in $G$ with respect to $v_0$. Moreover, $c_H$ can be constructed in polynomial time.
\end{theorem}

We also study \cVD~from the polyhedral point of view. In particular we investigate how well linear programming (LP) relaxations can approximate the optimal value of \cVD. As in~\cite{chan2016approximate,BPZ,bazzi2019no}, we use the following notion of LP relaxations which, by design, allows for extended formulations. 

Fix a graph $G$. Let $d \in \Z_{\geqslant 0}$ be an arbitrary dimension. A system of linear inequalities $Ax \geqslant b$ in $\R^d$ defines an \emph{LP relaxation} of \cVD{} on $G$ if the following hold:
(i) For every hitting set $X \subseteq V(G)$, we have a point $\pi^X \in \R^d$ satisfying $A\pi^X \geqslant b$; (ii) For every cost function $c : V(G) \to \Q_{\geqslant 0}$, we have an affine function $f_c : \R^d \to \R$; (iii) For all hitting sets $X \subseteq V(G)$ and cost functions $c : V(G) \to \Q_{\geqslant 0}$, the condition $f_c(\pi^X) = c(X)$ holds. 

The \emph{size} of the LP relaxation $Ax \geqslant b$ is defined as the number of rows of $A$. 
For every cost function $c$, the quantity $\LP(G,c) := \min \{f_c(x) \mid Ax \geqslant b\}$ gives a lower bound on $\OPT(G,c)$. The \emph{integrality gap} of the LP relaxation $Ax \geqslant b$ is defined as $\sup \{ \OPT(G,c) / \LP(G,c) \mid c \in \Q_{\geqslant 0}^{V(G)}\}$.

Letting $\mathcal{P}_3(G)$ denote the collection of all vertex sets $\{u,v,w\}$ that induce a $P_3$ in $G$, we define $$P(G) := \{ x \in [0,1]^{V(G)} \mid \forall \{u,v,w\} \in \mathcal{P}_3(G) : x_u + x_v + x_w \geqslant 1 \}.$$
We let $\SA_r(G)$ denote the relaxation obtained from $P(G)$ by applying $r$ rounds of the \emph{Sherali-Adams hierarchy}~\cite{SA1990}, a standard procedure to derive strengthened LP relaxations of binary linear programming problems. If a cost function $c : V(G) \to \Q_{\geqslant 0}$ is provided, we let $$\SA_r(G,c) := \min \{\sum_{v \in V(G)} c(v) x_v \mid x \in \SA_r(G)\}$$
denote the optimum value of the corresponding linear programming relaxation. 

It is not hard to see that the straightforward LP relaxation $P(G)$ has worst case integrality gap equal to $3$ (by \emph{worst case}, we mean that we take the supremum over all graphs $G$). Indeed, for a random $n$-vertex graph, $\OPT(G,\mathbf{1}_G) = n - O(\log^2 n)$ with high probability.  This can be easily proved via the probabilistic method.  A similar proof can be found, for instance, in the introduction of \cite{alonspencer2016}).  On the other hand, $\LP(G,\mathbf{1}_G) \leqslant n/3$, since the vector with all coordinates equal to $\frac{1}{3}$ is feasible for $P(G)$.

On the positive side, we show how applying one round of the Sherali-Adams hierarchy gives a relaxation with integrality gap at most $5/2 = 2.5$, see Theorem~\ref{thm:SA1CVD_UB}. To complement this, we prove that the worst case integrality gap of the relaxation is precisely $5/2$, see Theorem~\ref{thm:SA1CVD_LB}. We then show that the integrality gap decreases to $2+\eps$ after applying $\mathrm{poly}(1/\eps)$ rounds, see Theorem~\ref{thm:SA_2+eps}. 




On the negative side, applying known results on \VC~\cite{bazzi2019no}, we show that no polynomial-size LP relaxation of \cVD{} can have integrality gap at most $2 - \eps$ for some $\eps > 0$. As is the case for similar lower bounds (see \cite{rothvoss13, aprile2017extension}), this result is unconditional: it does not rely on P $\neq$ NP nor the Unique Games Conjecture. 


\subsection{Comparison to previous works} 

We now revisit all previous approximation algorithms for \cVD~\cite{YWC17,fiorini2016improved, FJS20}.
The presentation given here slightly departs from~\cite{YWC17,fiorini2016improved}, and explains in a unified manner what is the bottleneck in each of the algorithms.

Fix $k \in \{3,4,5\}$, and let $\alpha := (2k-1)/(k-1)$. Notice that $\alpha = 5/2$ if $k = 3$, $\alpha = 7/3$ if $k = 4$ and $\alpha = 9/4$ if $k = 5$. In~\cite[Lemma 3]{FJS20},
it is shown that if a twin-free graph $G$ contains a $k$-clique, then one can find an induced subgraph $H$ containing the $k$-clique and a cost function $c_H$ such that $(H,c_H)$ is strongly $\alpha$-good.

Therefore, in order to derive an $\alpha$-approximation for \cVD, one may assume without loss of generality that the input graph $G$ is twin-free and has no $k$-clique. Let $v_0$ be a maximum degree vertex in $G$, and let $H$ denote the subgraph of $G$ induced by $N_{\leqslant 2}[v_0]$. In~\cite{FJS20}, it is shown by a tedious case analysis that one can construct a cost function $c_H$ such that $(H,c_H)$ is $2$-good in $G$, using the fact that $G$ has no $k$-clique. 

The simplest case occurs when $k = 3$. Then $N(v_0)$ is a stable set. Letting $c_H(v_0) := d(v_0)-1$, $c_H(v) := 1$ for $v \in N(v_0)$ and $c_H(v) := 0$ for the other vertices of $H$, one easily sees that $(H,c_H)$ is (centrally) $2$-good in $G$ . For higher values of $k$, one has to work harder. 

In this paper, we show that one can \emph{always}, and in polynomial time, construct a cost function $c_H$ on the vertices at distance at most $2$ from $v_0$ that makes $(H,c_H)$ $2$-good in $G$, provided that $G$ is twin-free, see Theorem~\ref{thm:localratio}. This result was the main missing ingredient in previous approaches, and single-handedly closes the approximability status of \cVD.

\subsection{Other related works} \label{sec:otherrelated}

\cVD{} has also been widely studied from the perspective of \emph{fixed parameter tractability}. Given a graph $G$ and parameter $k$ as input, the task is to decide if $G$ has a hitting set $X$ of size at most $k$. A $2^k n^{\mathcal{O}(1)}$-time algorithm for this problem was given by H\"{u}ffner, Komusiewicz, Moser, and Niedermeier~\cite{HKMN10}. This was subsequently improved to a $1.911^k n^{\mathcal{O}(1)}$-time algorithm by Boral, Cygan, Kociumaka, and Pilipczuk~\cite{BCKP16}, and a $1.811^k n^{\mathcal{O}(1)}$-time algorithm by Tsur~\cite{Tsur19}.  By the general framework of Fomin, Gaspers, Lokshtanov, and Saurabh~\cite{FGLS19}, these parametrized algorithms can be transformed into exponential algorithms which compute the size of a minimum hitting set for $G$ exactly, the fastest of which runs in time $\mathcal{O}(1.488^n)$.

For polyhedral results, Hosseinian and Butenko~\cite{hosseinian2021polyhedral} gives some facet-defining inequalities of the \cVD{} polytope, as well as complete linear descriptions for special classes of graphs.

Another related problem is the \emph{feedback vertex set} problem in tournaments (\FVST).  Given a tournament $T$ with costs on the vertices, the task is to find a minimum cost set of vertices $X$ such that $T-X$ does not contain a directed cycle. 

For unit costs, note that \cVD{} is equivalent to the problem of deleting as few elements as possible from a \emph{symmetric} relation to obtain a transitive relation, while \FVST{} is equivalent to the problem of deleting as few elements as possible from an \emph{antisymmetric} and complete relation to obtain a transitive relation.

In a tournament, hitting all directed cycles is equivalent to hitting all directed triangles, so \FVST{} is also a hitting set problem in a $3$-uniform hypergraph. Moreover, \FVST{} is also UGC-hard to approximate to a constant factor smaller than $2$.  Cai, Deng, and Zang~\cite{CDZ01} gave a $5/2$-approximation algorithm for \FVST, which was later improved to a $7/3$-approximation algorithm by Mnich, Williams, and V\'{e}gh~\cite{MWV16}.  Lokshtanov, Misra, Mukherjee, Panolan, Philip, and Saurabh~\cite{LMMPPS20} recently gave a \emph{randomized} $2$-approximation algorithm, but no deterministic (polynomial-time) $2$-approximation algorithm is known. For \FVST{}, one round of the Sherali-Adams hierarchy actually provides a $7/3$-approximation~\cite{aprile2020simple}. 

Among other related covering and packing problems, Fomin, Le, Lokshtanov, Saurabh, Thomass\'e, and Zehavi~\cite{FominLLSTZ19} studied both \cVD{} and \FVST{} from the \emph{kernelization} perspective.  They proved that the unweighted versions of both problems admit subquadratic kernels: $\mathcal{O}(k^{\frac{5}{3}})$ for \cVD{} and $\mathcal{O}(k^{\frac{3}{2}})$ for $\FVST$.

\subsection{Overview of the proof} We give a sketch of the proof of Theorem~\ref{thm:localratio}. Recall that $H=G[N_{\leqslant 2}[v_0]]$. 
If the subgraph induced by $N(v_0)$ contains a hole (that is, an induced cycle of length at least $4$), then $H$ is strongly $2$-good by Lemma~\ref{lem:wheel}. 
If the subgraph induced by $N(v_0)$ contains an induced $2P_3$ (that is, two disjoint copies of $P_3$ with no edges between each other), then $H$ is strongly $2$-good by Lemma~\ref{lem:2p3free}. 
This allows us to reduce to the case where the subgraph induced by $N(v_0)$ is chordal and $2P_3$-free. 

Lemma~\ref{lem:hittingclique2good} then gives a direct construction of a cost function $c_H$ which certifies that $(H,c_H)$ is centrally $2$-good, provided that the subgraph induced by $N[v_0]$ is twin-free. This is the crucial step of the proof. It serves as the base case of the induction. Here, we use a slick observation due to Lokshtanov~\cite{DanielPC}: since the subgraph induced by $N(v_0)$ is chordal and $2P_3$-free, it has a hitting set that is a clique. In a previous version, our proof of Theorem~\ref{thm:localratio} was slightly more complicated.

We show inductively that we can reduce to the case where the subgraph induced by $N[v_0]$ is twin-free. The idea is to delete vertices from $H$ to obtain a smaller graph $H'$, while preserving certain properties, and then compute a suitable cost function $c_H$ for $H$, given a suitable cost function $c_{H'}$ for $H'$. We delete vertices at distance $2$ from $v_0$. When this creates twins in $H$, we delete one vertex from each pair of twins. At the end, we obtain a twin-free induced subgraph of $H[N[v_0]]$, which corresponds to our base case.

We conclude the introduction with a brief description of the different sections of the paper.  Section~\ref{sec:finding} is entirely devoted to the proof of Theorem~\ref{thm:localratio}. The proof of Theorem~\ref{thm:2apx} is given in Section~\ref{sec:analysis}, together with a complexity analysis of Algorithm~\ref{algo}. Section~\ref{sec:polyhedral} presents our polyhedral results. A conclusion is given in Section~\ref{sec:conclusion}. There, we state a few open problems for future research.

\smallskip

\noindent {\bf Conference version.} An abstract of this paper appeared in the proceedings of IPCO 2021~\cite{aprile2021tight}. The current paper is an extended version with full proofs and detailed discussions. In particular, the running time analysis of our algorithm in Section~\ref{sec:analysis} is new, and the polyhedral results of Section \ref{sec:polyhedral} appear without proof in the conference version.


\section{Finding $2$-good induced subgraphs} \label{sec:finding}

The goal of this section is to prove Theorem~\ref{thm:localratio}. Our proof is by induction on the number of vertices in $H := G[N_{\leq 2}[v_0]]$. First, we quickly show that we can assume that the subgraph induced by $N(v_0)$ is chordal and $2P_3$-free. Using this, we prove the theorem in the particular case where the subgraph induced by $N[v_0]$ is twin-free. Finally, we prove the theorem in the general case by showing how to deal with twins.

\subsection{Restricting to chordal, $2P_3$-free neighborhoods} 

A vertex of $G$ is  \emph{apex} if it is adjacent to all the other vertices of $G$. A \emph{wheel} is a graph obtained from a cycle by adding an apex vertex (called the \emph{center}).

As pointed out earlier in the introduction, $4$-cycles are strongly $2$-good. This implies that the wheel on $5$ vertices is strongly $2$-good (putting a zero cost on the center). We now show that \emph{all} wheels on at least $5$ vertices are strongly $2$-good. This allows our algorithm to restrict to input graphs such that the subgraph induced on each neighborhood is chordal. In a similar way, we show that we can further restrict such neighborhoods to be $2P_3$-free.

\begin{lemma} \label{lem:wheel}
Let $H := W_k$ be a wheel on $k \geqslant 5$ vertices and center $v_0$, let $c_H(v_0) := k-5$ and $c_H(v) := 1$ for $v \in V(H-v_0)$. Then $(H,c_H)$ is strongly $2$-good.
\end{lemma}

\begin{proof}
Notice that $\OPT(H,c_H) \geqslant k-3$ since a hitting set either contains $v_0$ and at least $2$ more vertices, or does not contain $v_0$ but contains $k-3$ other vertices. Hence, $\sum_{v \in V(H)} c_H(v) = k-5 + k-1 = 2(k-3) \leqslant 2 \OPT(H,c_H)$.
\end{proof}

\begin{lemma}\label{lem:2p3free}
Let $H$ be the graph obtained from $2P_3$ by adding an apex vertex $v_0$. Let $c_H(v_0) := 2$ and $c_H(v) := 1$ for $v \in V(H-v_0)$. Then $(H,c_H)$ is strongly $2$-good.
\end{lemma}
\begin{proof}
It is easy to check that $\OPT(H,c_H) \geqslant 4$. Thus, $\sum_{v \in V(H)} c_H(v) = 8 \leqslant 2 \OPT(H,c_H)$. 
\end{proof}

\subsection{When $H$ is twin-free}

Throughout this section, we assume that $H$ is a twin-free graph with an apex vertex $v_0$ such that $H-v_0$ is chordal and $2P_3$-free. Our goal is to construct a cost function $c_H$ that certifies that $H$ is centrally $2$-good. 

It turns out to be easier to define the cost function on $V(H-v_0)=N(v_0)$ first, and then adjust the cost of $v_0$. This is the purpose of the next lemma. Below, $\omega(G,c)$ denotes the maximum weight of a clique in weighted graph $(G,c)$. 

\begin{lemma} \label{lem:without_v0}
Let $H$ be a graph with an apex vertex $v_0$ and $H':=H-v_0$. Let $c_{H'} : V(H') \rightarrow \Z_{\geqslant 1}$ be a cost function such that
	\begin{enumerate}[(i)]
    \item \label{it:c2} $c_{H'}(V(H')) \geqslant 2\omega(H',c_{H'})$ and
    \item \label{it:c3} $\OPT(H',c_{H'}) \geqslant \omega(H',c_{H'}) - 1$.
	\end{enumerate}
Then we can extend $c_{H'}$ to a function $c_H : V(H) \rightarrow \Z_{\geqslant 1}$ such that $c_H(V(H)) \leqslant 2 \OPT(H,c_H) + 1$. In other words, $(H,c_H)$ is centrally $2$-good with respect to $v_0$.
\end{lemma}
\begin{proof}
Notice that 
$$\OPT(H,c_H)=\min(c_H(v_0)+\OPT(H',c_{H'}), c_{H'}(V(H'))-\omega(H',c_{H'})),$$ since if $X$ is hitting set of $H$ that does not contain $v_0$, then $H-X$ is a clique.

Let $a=\max(1,c_{H'}(V(H'))-2\OPT(H',c_{H'})-1)$ and $b=c_{H'}(V(H'))-2\omega(H',c_{H'})+1$.  Choose $c_H(v_0) \in \Z_{\geqslant 1}$ such that $a \leqslant c_H(v_0) \leqslant b$.
Note that $c_H(v_0)$ exists since $a \leqslant b$ by conditions \eqref{it:c2} and \eqref{it:c3}.

Suppose $\OPT(H,c_H) = c_H(v_0)+\OPT(H',c_{H'})$.  Since $a \leqslant c_H(v_0)$, $$c_{H}(V(H)) \leqslant 2c_H(v_0)+2\OPT(H',c_{H'})+1=2\OPT(H,c_H)+1.$$ 
Suppose $\OPT(H,c_H) = c_{H'}(V(H'))-\omega(H',c_{H'})$.  Since $c_H(v_0) \leqslant b$, $$c_H(V(H)) \leqslant 2c_{H'}(V(H'))-2\omega(H',c_{H'})+1=2\OPT(H,c_H)+1.$$  In either case, $c_H(V(H)) \leqslant 2\OPT(H,c_H)+1$, as required.
\end{proof}


We abuse notation and regard a clique $X$ of a graph as both a set of vertices and a subgraph.
We call a hitting set $X$ of a graph $G$ a \emph{hitting clique} if $X$ is also a clique. 

\begin{lemma}\label{lem:chordal+2p3free}
Every chordal, $2P_3$-free graph contains a hitting clique.
\end{lemma}
\begin{proof}
Let $G$ be a chordal, $2P_3$-free graph. Since $G$ is chordal, $G$ admits a \emph{clique tree} $T$ (see~\cite{blair1993introduction}). In $T$, the vertices are the maximal cliques of $G$ and, for every two maximal cliques $K$, $K'$, the intersection $K\cap K'$ is contained in every clique of the $K$--$K'$ path in $T$. For an edge $e:=KK'$ of $T$, Let $T_1$ and $T_2$ be the components of $T - e$ and $G_1$ and $G_2$ be the subgraphs of $G$ induced by the union of all the cliques in $T_1$ and $T_2$, respectively.  It is easy to see that deleting $K \cap K'$ separates $G_1$ from $G_2$ in $G$.  Now, since $G$ is $2P_3$-free, at least one of $G_1':=G_1 -(K \cap K')$ or $G_2':=G_2 - (K \cap K')$ is a cluster graph. If both $G_1'$ and $G_2'$ are cluster graphs, we are done since $K\cap K'$ is the desired hitting clique. Otherwise, if $G_i'$ is not a cluster graph, then we can orient $e$ towards $T_i$.  Applying this argument on each edge, we define an orientation of $T$, which must have a sink $K_0$. But then removing $K_0$ from $G$ leaves a cluster graph, and we are done. Since the clique tree of a chordal graph can be constructed in polynomial time~\cite{blair1993introduction}, the hitting clique can be found in polynomial time.  
\end{proof}

\begin{figure}[h!]
\centering
\begin{tikzpicture}[inner sep=3pt,scale=.75]
\tikzstyle{vtx}=[circle,draw,thick];
\tikzstyle{vtxg}=[circle,draw,thick,fill=black!50];
\tikzstyle{vtxr}=[circle,draw,thick,fill=red!50];
\tikzstyle{vtxp}=[circle,draw,thick,fill=pink!50];
\tikzstyle{vtxb}=[circle,draw,thick,fill=blue!50];
\node[vtxb] (v3) at (-4.5,-.5) {};
\node[vtxb] (v1) at (-4,.5) {};
\node[vtxb] (v2) at (-6,1) {};
\node[vtxb] (v4) at (-3,1.0) {};
\node[vtx] (v5) at (-3.5,2.0) {};
\node[vtx] (v6) at (-4.5,2.5) {};
\node[vtx] (v7) at (-5.5,-.2) {};
\node[vtxg] (v0) at (-5.5,2) {};
\draw[thick,-,>=latex] (v3) -- (v1);
\draw[thick,-,>=latex] (v1) -- (v2);
\draw[thick,-,>=latex] (v2) -- (v0);
\draw[thick,-,>=latex] (v0) -- (v3);
\draw[thick,-,>=latex] (v2) -- (v4);
\draw[thick,-,>=latex] (v1) -- (v4);
\draw[thick,-,>=latex] (v2) -- (v3);
\draw[thick,-,>=latex] (v2) -- (v5);
\draw[thick,-,>=latex] (v3) -- (v4);
\draw[thick,-,>=latex] (v0) -- (v5);
\draw[thick,-,>=latex] (v0) -- (v4);
\draw[thick,-,>=latex] (v0) -- (v1);
\draw[thick,-,>=latex] (v0) -- (v6);
\draw[thick,-,>=latex] (v1) -- (v6);
\draw[thick,-,>=latex] (v3) -- (v7);
\draw[thick,-,>=latex] (v0) -- (v7);
\node at (-3,0) {$H$};
\node[vtxb] (w3) at (1,-.5) {};
\node[vtxb] (w1) at (1.5,.5) {};
\node[vtxb] (w2) at (-0.5,1) {};
\node[vtxb] (w4) at (2.5,1) {};
\node[vtx] (w5) at (2,2.0) {};
\node[vtx] (w6) at (1.0,2.5) {};
\node[vtx] (w7) at (0,-.2) {};
\draw[thick,-,>=latex] (w3) -- (w4);
\draw[thick,-,>=latex] (w2) -- (w4);
\draw[thick,-,>=latex] (w1) -- (w4);
\draw[thick,-,>=latex] (w3) -- (w1);
\draw[thick,-,>=latex] (w1) -- (w2);
\draw[thick,-,>=latex] (w2) -- (w3);
\draw[thick,-,>=latex] (w2) -- (w5);
\draw[thick,-,>=latex] (w1) -- (w6);
\draw[thick,-,>=latex] (w3) -- (w7);
\node at (2.5,0) {$H-v_0$};
\node at (7.5,1.5){$v$};
\node[vtxb] (z4) at (7.5,1.0) {};
\node[vtx] (z5) at (7.0,2.0) {};
\node[vtx] (z6) at (6.0,2.5) {};
\node[vtx] (z7) at (5.0,-.2) {};
\node at (7.5,0) {$S_v$};
\end{tikzpicture}
\caption{Here $H$ is twin-free, $v_0$ is the gray vertex and the blue vertices form a hitting clique $K_0$ for $H-v_0$, which is chordal and $2P_3$-free. For $v\in K_0$, the set $S_v$ defined as in the proof of Lemma \ref{lem:hittingclique2good} consists of the unique maximal independent set containing $v$. We obtain $c_{H} = (\mathbf{6},{\color{blue}{1}},{\color{blue}{1}},{\color{blue}{1}},{\color{blue}{1}},3,3,3)$, which is easily seen to be centrally 2-good with respect to $v_0$.}
\label{fig:hittingclique}
\end{figure}
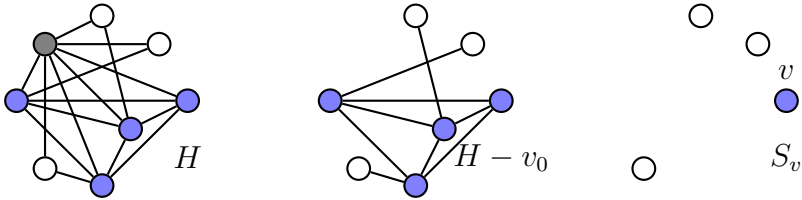

We are ready to prove the base case for Theorem~\ref{thm:localratio}. For a graph $H$, we let $|H|$ denote the number of vertices of $H$.

\begin{lemma} \label{lem:hittingclique2good}
Let $H$ be a twin-free graph with an apex vertex $v_0$ such that $H-v_0$ is chordal and $2P_3$-free. There exists a cost function $c_H$ such that $(H,c_H)$ is centrally $2$-good with respect to $v_0$. Moreover, $c_H$ can be found in time $\mathcal{O}(|H|^3)$.  
\end{lemma}
\begin{proof}
By Lemma~\ref{lem:chordal+2p3free}, some maximal clique of $H - v_0$, say $K_0$, is a hitting set.

We claim that there is a family of stable sets $\mathcal{S}=\{S_v \mid v\in K_0\}$ of $H-v_0$ satisfying the following properties:
\begin{enumerate}[({P}1)]
    \item \label{P:coverage} every vertex of $H-v_0$ is contained in some $S_v$;
    \item \label{P:size_at_least_2} for each $v \in K_0$, $S_v$ contains $v$ and at least one other vertex;
    \item \label{P:every_two_induce_P3} for every two distinct vertices $v, v'\in K_0$, $H[S_v\cup S_{v'}]$ contains a $P_3$.
\end{enumerate}

Before proving the claim, we prove that it implies the lemma. Consider the cost function $c_{H'} := \sum_{v\in K_0} \chi^{S_v}$ on the vertices of $H':=H-v_0$ defined by giving to each vertex $u$ a cost equal to the number of stable sets $S_v$ that contain $u$ (see Figure \ref{fig:hittingclique}). It suffices to show that $c_{H'}$ satisfies the conditions of Lemma~\ref{lem:without_v0} and can therefore be extended to a cost function $c_H$ on $V(H)$ such that $(H, c_H)$ is centrally $2$-good with respect to $v_0$.  

First, by~(P\ref{P:coverage}), we have $c_{H'}(u) \in \Z_{\geqslant 1}$ for all $u \in V(H')$. Second, condition (\ref{it:c2}) of Lemma~\ref{lem:without_v0} follows from~(P\ref{P:size_at_least_2}) since each stable set $S_v$ contributes at least two units to $c_{H'}(V(H'))$ and at most one unit to $\omega(H',c_{H'})$. Third,~(P\ref{P:every_two_induce_P3}) implies that every hitting set of $H'$ meets every stable set $S_v$, except possibly one. Hence, $\OPT(H',c_{H'}) \geqslant |K_0|-1$. Also, every clique of $H'$ meets every stable set $S_v$ in at most one vertex, implying that $\omega(H',c_{H'}) \leqslant |K_0|$, and equality holds since $c_{H'}(K_0)=|K_0|$. Putting the last two observations together, we see that $\OPT(H',c_{H'}) \geqslant |K_0|-1 = \omega(H',c_{H'})-1$ and hence condition \eqref{it:c3} of Lemma~\ref{lem:without_v0} holds.

Now, we prove that our claim holds. Let $K_1, \dots, K_t$ denote the clusters (maximal cliques) of cluster graph $H - v_0 - K_0$. For $i \in [t]$, consider the submatrix $A_i$ of the adjacency matrix $A(H)$ with rows indexed by the vertices of $K_0$ and  columns indexed by the vertices of $K_i$. 

Notice that $A_i$ contains neither $\begin{pmatrix} 1& 0\\ 0& 1 \end{pmatrix}$ nor $\begin{pmatrix} 0& 1\\ 1& 0 \end{pmatrix}$ as a submatrix, as this would give a $C_4$ contained in $H - v_0$, contradicting the chordality of $H - v_0$. Hence, after permuting its rows and columns if necessary, $A_i$ can be assumed to be staircase-shaped. That is, every row of $A_i$ is nonincreasing and every column nondecreasing. Notice also that $A_i$ does not have two equal columns, since these would correspond to two vertices of $K_i$ that are twins. 

For each $K_i$ that is not complete to $v \in K_0$, define $\varphi_i(v)$ as the vertex $u \in K_i$ whose corresponding column in $A_i$ is the \emph{first} containing a $0$ in row $v$. Now, for each $v$, let $S_v$ be the set including $v$, and $\varphi_i(v)$, for each $K_i$ that is not complete to $v$. 

Because $K$ is maximal, no vertex $u \in K_i$ is complete to $K_0$. Since no two columns of $A_i$ are identical, we must have $u = \varphi_i(v)$ for some $v \in K_0$. This proves~(P\ref{P:coverage}). 

Notice that $v \in S_v$ by construction and that $|S_v| \geqslant 2$ since otherwise, $v$ would be apex in $H$ and thus a twin of $v_0$. Hence,~(P\ref{P:size_at_least_2}) holds. 

Finally, consider any two distinct vertices $v, v' \in K_0$. Since $v, v'$ are not twins, the edge $vv'$ must be in a $P_3$ contained in $H - v_0$. Assume, without loss of generality, that there is a vertex $u \in K_i$ adjacent to $v$ and not to $v'$ for some $i \in [t]$. Then $\{v,v',\varphi_i(v')\}$ induces a $P_3$ contained in $H[S_v\cup S_{v'}]$, proving~(P\ref{P:every_two_induce_P3}). This concludes the proof of the claim.

We observe that the collection $\mathcal{S}$ can be computed in $\mathcal{O}(|H|^3)$ time.  This yields the 
restriction $c_H$ to $H'$. Since $H'$ is chordal, $\omega(H',c_{H'})$ can be computed in $\mathcal{O}(|H'|^2)$ time.  We then just let $c_H(v_0) := c_{H'}(V(H'))-2\omega(H',c_{H'})+1=c_{H'}(V(H'))-2|K_0|+1$. This sets $c_H(v_0)$ equal to the value $b$ in the proof of Lemma~\ref{lem:without_v0}.  Therefore, $c_H$ can be constructed in $\mathcal{O}(|H|^3)$ time.
\end{proof}

\subsection{Handling twins in $G[N[v_0]]$}\label{sec:truetwins}

We now deal with the general case where $G[N[v_0]]$ contains twins.
We start with an extra bit of terminology relative to twins. Let $G$ be a twin-free graph, and $v_0 \in V(G)$. Suppose that $u, u'$ are twins in $G[N[v_0]]$. Since $G$ is twin-free, there exists a vertex $v$ that is adjacent to exactly one of $u$, $u'$ in $G$. We say that $v$ is a \emph{distinguisher} for the edge $uu'$ (or for the pair $\{u,u'\}$). Notice that either $uu'v$ or $u'uv$ is an induced $P_3$. Notice also that $v$ is at distance $2$ from $v_0$. 

Now, consider a graph $H$ with a special vertex $v_0 \in V(H)$ (the \emph{root} vertex) such that 
\begin{enumerate}[({H}1)]
\item \label{H:radius_at_most_2} every vertex is at distance at most $2$ from $v_0$, and 
\item \label{H:distinguished} every pair of vertices that are twins in $H[N[v_0]]$ has a distinguisher.
\end{enumerate}

Let $v$ be any vertex that is at distance $2$ from $v_0$. Consider the equivalence relation $\equiv$ on $N[v_0]$ with $u \equiv u'$ whenever $u = u'$ or $u, u'$ are twins in $H - v$. Observe that the equivalence classes of $\equiv$ are of size at most $2$ since, if $u, u', u''$ are distinct vertices with $u \equiv u' \equiv u''$, then $v$ cannot distinguish every edge of the triangle on $u$, $u'$ and $u''$. Hence, two of these vertices are twins in $H$, which contradicts (H\ref{H:distinguished}).

From what precedes, the edges contained in $N[v_0]$ that do not have a distinguisher in $H-v$ form a matching $M := \{u_1u'_1,\ldots,u_ku'_k\}$ (possibly, $k = 0$). Let $H'$ denote the graph obtained from $H$ by deleting $v$ and exactly one endpoint from each edge of $M$. Notice that the resulting subgraph is the same, up to isomorphism, no matter which endpoints are chosen. 

The lemma below states how we can obtain a cost function $c_H$ that certifies that $H$ is centrally $2$-good from a cost function $c_{H'}$ that certifies that $H'$ is centrally $2$-good. It is inspired by~\cite[Lemma 3]{FJS20}. See Figure~\ref{fig:rule2} for an example. 

\begin{lemma} \label{lem:get_rid_of_true_twins}
Let $H$ be any graph satisfying (H\ref{H:radius_at_most_2}) and (H\ref{H:distinguished}) for some $v_0 \in V(H)$. Let $v \in N_2(v_0)$. Let $M := \{u_1u'_1,\ldots,u_ku'_k\}$ be the matching formed by the edges in $N[v_0]$ whose unique distinguisher is $v$, where $u'_i \neq v_0$ for all $i$ (we allow the case $k = 0$). Let $H' := H - u'_1 - \dots - u'_k - v$. Given a cost function $c_{H'}$ on $V(H')$, define a cost function $c_{H}$ on $V(H)$ by letting $c_H(u'_i) := c_{H'}(u_i)$ for $i \in [k]$, $c_H(v) := \sum_{i = 1}^k c_{H'}(u_i) = \sum_{i=1}^k c_{H}(u'_i)$, and $c_{H}(u) := c_{H'}(u)$ otherwise. First, $H'$ satisfies (H\ref{H:radius_at_most_2}) and (H\ref{H:distinguished}). Second, if $(H',c_{H'})$ is centrally $2$-good, then $(H,c_{H})$ is centrally $2$-good.
\end{lemma}

\begin{proof}
For the first part, notice that $H'$ satisfies (H\ref{H:radius_at_most_2}) by our choice of $v$. Indeed, deleting $v$ does not change the distance of the remaining vertices from $v_0$. 

We now prove that $H'$ satisfies (H\ref{H:distinguished}). First notice that in $H-v$ the twins in $N[v_0]$ are exactly $(u_1,u'_1),\ldots,(u_k,u'_k)$.  Next, for each edge $e\in E(H[N[v_0]])$, $e$ has at least one distinguisher different than $v$ unless $e\in M$. Moreover, each $u_i'$ is a distinguisher for $e$ if and only if $u_i$ is a distinguisher for $e$. Thus, each edge of $H'[N[v_0]]$ still has at least one distinguisher, which proves (H\ref{H:distinguished}).

For the second part, notice that $c_H(u) \geqslant 1$ for all $u \in N[v_0]$ since $c_{H'}(u) \geqslant 1$ for all $u \in N[v_0] \setminus \{v,u'_1,\ldots,u'_k\}$. To argue that $c_H(V(H)) \leqslant 2 \OPT(H,c_H) + 1$, one can check that any hitting set of $H$ must either contain $v$ or at least one endpoint of each edge $u_iu_i' \in M$. Hence $\OPT(H,c_H) \geqslant \sum_{i=1}^{k}c_{H'}(u_i) + \OPT(H',c_{H'})$. 

Since $(H,c_{H'})$ is centrally $2$-good, $c_{H'}(V(H')) \leqslant 2 \OPT(H',c_{H'}) + 1$. It follows that
\begin{align*}
c_{H}(V(H)) &= \overbrace{c_{H}(v) + \sum_{i=1}^{k}c_{H}(u'_i)}^{= 2 \sum_{i=1}^{k}c_{H'}(u_i)} + \overbrace{c_{H'}(V(H'))}^{\leqslant 2 \OPT(H',c_{H'}) + 1}\\
&\leqslant 2\OPT(H,c_H) + 1\,. 
\end{align*}
\let\qed\relax
\end{proof}

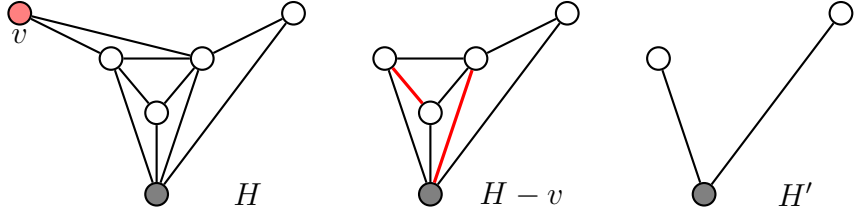
\begin{figure}[h!]
\centering
\begin{tikzpicture}[inner sep=3pt,scale=.6]
\tikzstyle{vtx}=[circle,draw,thick];
\tikzstyle{vtxg}=[circle,draw,thick,fill=black!50];
\tikzstyle{vtxr}=[circle,draw,thick,fill=red!50];
\node[vtxr] (v5) at (-3,1) {};
\node at (-3,.5){$v$};
\node[vtx] (v3) at (1,0) {};
\node[vtx] (v1) at (-1,0) {};
\node[vtx] (v2) at (0,-1.2) {};
\node[vtx] (v4) at (3,1) {};
\node[vtxg] (v0) at (0,-3) {};
\draw[thick] (v0) -- (v2);
\draw[thick] (v3) -- (v5);
\draw[thick] (v0) -- (v1);
\draw[thick] (v0) -- (v3);
\draw[thick] (v0) -- (v4);
\draw[thick] (v2) -- (v1);
\draw[thick] (v1) -- (v3);
\draw[thick] (v2) -- (v3);
\draw[thick] (v5) -- (v1);
\draw[thick] (v4) -- (v3);
\node at (2,-3) {$H$};
\node[vtx] (w3) at (7,0) {};
\node[vtx] (w1) at (5,0) {};
\node[vtx] (w2) at (6,-1.2) {};
\node[vtx] (w4) at (9,1) {};
\node[vtxg] (w0) at (6,-3) {};
\draw[thick] (w0) -- (w2);
\draw[thick] (w0) -- (w1);
\draw[very thick,red] (w0) -- (w3);
\draw[thick] (w0) -- (w4);
\draw[very thick,red] (w2) -- (w1);
\draw[thick] (w1) -- (w3);
\draw[thick] (w2) -- (w3);
\draw[thick] (w4) -- (w3);
\node at (8,-3) {$H-v$};
\node[vtx] (z1) at (11,0) {};
\node[vtx] (z4) at (15,1) {};
\node[vtxg] (z0) at (12,-3) {};
\draw[thick] (z0) -- (z1);
\draw[thick] (z0) -- (z4);
\node at (14,-3) {$H'$};
\end{tikzpicture}
\caption{Here $v_0$ is the gray vertex and $v$ is the red vertex. $H - v$ violates (H\ref{H:distinguished}), and contains two pairs of twins, indicated by the red edges. Lemma \ref{lem:get_rid_of_true_twins} applies. We see that $H'$ is a $P_3$, for which Lemma~\ref{lem:hittingclique2good} gives $c_{H'} = \mathbf{1}_{H'}$. In $(H,c_H)$, all vertices get a unit cost except $v$, which gets a cost of $2$, since there are $2$ pairs of twins in $H - v$. Thus, $c_H = (\mathbf{1},{\color{red}{2}},1,1,1,1)$, where the entries corresponding to $v_0$ and $v$ are bold and red, respectively.} 
\label{fig:rule2}
\end{figure}

\subsection{Proof of Theorem~\ref{thm:localratio}}\label{sec:putting}

We are ready to prove Theorem~\ref{thm:localratio}.

\begin{proof}
We can decide in polynomial time (see for instance \cite{tarjan1984}) if $H[N(v_0)]$ is chordal, and if not, output a hole of $H[N(v_0)]$. If the latter holds, we are done by Lemma \ref{lem:wheel}.  If the former holds, we can decide in polynomial time (see \cite{tarjan1985decomposition}, and the proof of Lemma \ref{lem:chordal+2p3free}) 
whether $H$ contains a $2P_3$. If it does, we are done by Lemma \ref{lem:2p3free}. 

From now on, assume that the subgraph induced by $N(v_0)$ is chordal and $2P_3$-free. This is done without loss of generality. Notice that hypotheses (H\ref{H:radius_at_most_2}) and (H\ref{H:distinguished}) from Section \ref{sec:truetwins} hold for $H$. This is obvious for (H\ref{H:radius_at_most_2}). To see why (H\ref{H:distinguished}) holds, remember that $G$ is twin-free. Hence, every edge $uu'$ contained in $N[v_0]$ must have a distinguisher in $G$, which is in $N_{\leqslant 2}[v_0]$. (In fact, notice that if $u$ and $u'$ are twins in $H[N[v_0]]$ then the distinguisher is necessarily in $N_2(v_0)$.)

We repeatedly apply Lemma~\ref{lem:get_rid_of_true_twins} in order to delete each vertex of $N_2(v_0)$ one after the other and reduce to the case where $H$ is a twin-free graph for which $v_0$ is apex. We can then apply Lemma~\ref{lem:hittingclique2good}. The whole process takes polynomial time.
\end{proof}

\section{Running-time Analysis} \label{sec:analysis}

We now analyse the running-time of Algorithm~\ref{algo}.  We assume that input graphs are given by their adjacency matrix. We need the following easy lemma, whose proof we include for completeness. 

\begin{lemma} \label{lem:equalrows}
Given a matrix $N \in \{0,1\}^{r \times c}$, the set of all equivalence classes of equal rows of $N$ can be found in time $\mathcal{O}(r c)$. 
\end{lemma}

\begin{proof}
Let $R_0$ and $R_1$ be the set of rows of $N$ whose first entry is $0$ and $1$, respectively.  We can determine $R_0$ and $R_1$ by reading the first column of $N$, which takes time $O (r)$.  We then recurse on $N_0'$ and $N_1'$, where $N_i'$ is the submatrix of $N$ induced by $R_i$ and the last $c-1$ columns of $N$.
\end{proof}

Before proving the next lemma we remark that, given a graph $H$ with $n$ vertices and $m$ edges, one can check whether $H$ is a cluster graph by checking that each of its components is a clique, which takes $\mathcal{O}(n^2)$ time.

\begin{lemma}\label{lem:costtime}
Let $G$ be an $n$-vertex, twin-free graph. In $\mathcal{O}(n^3)$ time, we can find an induced subgraph $H$ of $G$ and a cost function $c_H$ on $V(H)$ such that $(H,c_H)$ is $2$-good.
\end{lemma}

\begin{proof}
We fix any vertex $v_0\in V(G)$, and let $H=G[N_{\leqslant 2}[v_0]]$. We can check in $\mathcal{O}(n^2)$ time whether $H[N(v_0)]$ is chordal by using the algorithm from \cite{tarjan1984}. If $H[N(v_0)]$ is not chordal this algorithm returns, as a certificate, a hole $C$. By Lemma~\ref{lem:wheel}, $H[V(C)+v_0]$ is strongly 2-good and the corresponding function $c_H$ can be computed straightforwardly, hence we are done. Suppose now that $H[N(v_0)]$ is chordal. We can construct the clique tree of $H[N(v_0)]$ (see for instance \cite{tarjan1985decomposition}) in $\mathcal{O}(n^2)$ time. Each edge of the tree induces a separation of $H[N(v_0)]$, and we can check if each side is a cluster graph in $\mathcal{O}(n^2)$ time.  If neither side is a cluster graph, then we have found a $2P_3$ in $H[N(v_0)]$.  Hence,  $H[V(2P_3)+v_0]$ is strongly 2-good and the corresponding function $c_H$ can be computed straightforwardly.  Since the clique tree has at most $|H|\leq n$ vertices, by orienting its edges as in the proof of Lemma~\ref{lem:chordal+2p3free} we find, in $\mathcal{O}(n^3)$ time, a hitting clique. By applying Lemma \ref{lem:get_rid_of_true_twins} to get rid of twins, which can be done in $\mathcal{O}(n^2)$ time, we obtain a subgraph satisfying the hypotheses of Lemma \ref{lem:hittingclique2good}. Finally, by Lemma \ref{lem:hittingclique2good}, the cost function in the statement of Lemma \ref{lem:hittingclique2good} can be constructed in $\mathcal{O}(n^3)$ time. 
\end{proof}

\begin{lemma}\label{lem:analysis}
Algorithm 1 runs in $\mathcal{O}(n^4)$-time.
\end{lemma}
\begin{proof}
By Lemma~\ref{lem:equalrows}, finding all twins in $G$ can be done in time $\mathcal{O} (n^2)$.  Therefore, the most expensive recursive call of the algorithm is the construction of the $2$-good weighted induced subgraph $(H,c_H)$ from Lemma~\ref{lem:costtime}, which can be done in time $\mathcal{O}(n^3)$.  Therefore, the running-time $T(n)$ of the algorithm satisfies $T(n) \leqslant T(n-1) + \mathcal{O}(n^3)$, which gives $T(n) = \mathcal{O}(n^4)$.
\end{proof}

\begin{proof}[of Theorem~\ref{thm:2apx}]
The proof is identical to~\cite[Proof of Theorem~1, pages 365--366]{FJS20}, except that factor $9/4$ needs to be replaced everywhere by $2$. One easily proves by induction that the vertex set $X$ output by the algorithm on input $(G,c)$ is a minimal hitting set with $c(X) \leqslant 2 \OPT(G,c)$. We do not include more details here, and instead refer the reader to~\cite{FJS20}. Theorem~\ref{thm:localratio} guarantees that Algorithm~\ref{algo} runs in polynomial time.
\end{proof}

\section{Polyhedral results} \label{sec:polyhedral}
In this section we study how well LP-relaxtions can approximate \cVD{}, as already described in Section \ref{sec:intro}.
We begin with a brief description of the Sherali-Adams hierarchy~\cite{SA1990}, which is a standard procedure to obtain strengthened LP relaxations for binary linear programs.  For a more thorough introduction, we refer the reader to~\cite{laurent03}. Throughout the section we closely follow the exposition given in \cite{aprile2020simple}, where the Sherali-Adams hierarchy and a very similar concept of diagonal (defined below) are used to approach the FVST problem, as mentioned in Section \ref{sec:otherrelated}. In particular, our Lemma \ref{lem:diagonal} is similar to another lemma in \cite{aprile2020simple}.

Let $P =  \{ x \in \R^n \mid A x \geqslant b \}$ be a polytope contained in $[0,1]^n$ and $P_I := \mathrm{conv}(P \cap \Z^n)$.  For each $r \in \mathbb N$, we define a polytope $P \supseteq \SA_1(P) \supseteq \dots \supseteq \SA_r(P) \supseteq P_I$ as follows.  Let $N_r$ be the \emph{nonlinear} system obtained from $P$ by multiplying each constraint by $\prod_{i \in I} x_i \prod_{j \in J}(1-x_j)$ for all disjoint subsets $I, J$ of $[n]$ such that $1 \leqslant |I|+|J| \leqslant r$. Note that if $x_i \in \{0,1\}$, then $x_i^2=x_i$.  Therefore, we can obtain a \emph{linear} system $L_r$ from $N_r$ by setting $x_i^2:=x_i$ for all $i \in [n]$ and then $x_I:=\prod_{i \in I} x_i$ for all $I \subseteq [n]$ with $|I| \geqslant 2$.  We then let $\SA_r(P)$ be the projection of $L_r$ onto the variables $x_i$, $i \in [n]$.  

Let $\mathcal{P}_3(G)$ denote the collection of all vertex sets $\{u,v,w\}$ that induce a $P_3$ in $G$ and let $\SA_r(G):=\SA_r(P(G))$, where
$$
P(G) := \{ x \in [0,1]^{V(G)} \mid \forall \{u,v,w\} \in \mathcal{P}_3(G) : x_u + x_v + x_w \geqslant 1 \}\,.
$$
If a cost function $c : V(G) \to \R_{\geqslant 0}$ is provided, we let
$$
\SA_r(G,c) := \min \left\{\sum_{v \in V(G)} c(v) x_v \mid x \in \SA_r(G)\right\}
$$
denote the optimum value of the corresponding linear programming relaxation. For the sake of simplicity, we sometimes denote by $\SA_r(G,c)$ the above linear program itself.  

We say vertices $a$ and $b$ form a \emph{diagonal} if there are vertices $u,v$ such that $\{u,v,a\} \in \mathcal{P}_3(G)$ and $\{u,v,b\} \in \mathcal{P}_3(G)$. We say that a path \emph{contains a diagonal} if any of its pairs of vertices are diagonals. Note that a diagonal pair in a path need not be an edge in the path. 

Our first results concern $\SA_1(G)$. For later use, we list here the inequalities defining $\SA_1(G)$. For all $\{u,v,w\} \in \mathcal{P}_3(G)$ and $z \in V(G-u-v-w)$, we have the inequalities
\begin{align}
\label{ineq:SAtype1} x_u + x_v + x_w &\geqslant 1 + x_{uv} + x_{vw}\,,\\
\label{ineq:SAtype2} x_{uz} + x_{vz} + x_{wz} &\geqslant x_z \quad \text{and}\\
\label{ineq:SAtype3} x_u + x_v + x_w + x_{z} &\geqslant 1 + x_{uz} + x_{vz} + x_{wz}\,.
\end{align}
In addition, there are the inequalities
\begin{equation}
\label{ineq:SAbounds}
1 \geqslant x_{v} \geqslant x_{vu} \geqslant 0
\end{equation}
for all distinct $u, v \in V(G)$. The polytope $\SA_1(G)$ is the set of all $(x_v)$ such that there exists $(x_{uv})$ such that inequalities \eqref{ineq:SAtype1}--\eqref{ineq:SAbounds} are satisfied. Note that by definition, $x_{uv}$ and $x_{vu}$ are the same variable. 

In order to establish the integrality gap of $\SA_1(G)$, we need two preliminary lemmas.

\begin{lemma}\label{lem:diagonal}
Let $x \in \SA_1(G)$. If $G$ contains a $P_3$ which has a diagonal, then $x_v \geqslant 2/5$ for some vertex $v$ of $G$.
\end{lemma}

\begin{proof}
Assume by contradiction that $a$, $b$ form a diagonal and all components of $x$ are less than 2/5. By the definition of diagonal, there exist $u,v \in V(G)$ with $\{u,v,a\}, \{u,v,b\} \in \mathcal{P}_3(G)$. In particular, from \eqref{ineq:SAtype1} we have $x_a + x_u + x_v \geqslant 1 + x_{au} + x_{av}$ and from \eqref{ineq:SAtype2} $x_{ab} + x_{au} + x_{av} \geqslant x_a$. 

Adding these two inequalities, we obtain $x_u + x_v + x_{ab} \geqslant 1$. We must have $x_{ab} \geqslant 1/5$ since otherwise $\max(x_u,x_v) \geqslant 2/5$. Now let $c$ be the third vertex of the $P_3$ containing $a, b$ (notice that $c$ can be the middle vertex of the $P_3$). By \eqref{ineq:SAtype1} and \eqref{ineq:SAbounds}, we have
$x_a + x_b + x_c \geqslant 1 + x_{ab} + x_{ac} \geqslant 6/5$ which means that $\max(x_a,x_b,x_c) \geqslant 2/5$, a contradiction which concludes the proof.
\end{proof}

 \begin{lemma}\label{lem:path-cycles}
 Let $G$ be a path or a cycle, and $c : V(G) \to \R_{\geqslant 0}$. Then:
 \begin{enumerate}[(i)]
     \item \label{it:path1} there is an efficient algorithm that solves \cVD{} for $(G,c)$;
     \item \label{it:path2} the basic LP $\SA_0(G)$ has integrality gap at most $2$.
 \end{enumerate}
 \end{lemma}
 \begin{proof}
 First, let $G$ be a path. We notice that the coefficient matrix of the basic LP is totally unimodular, by the consecutive ones property \cite{schrijver98}. Hence solving the LP yields an integral optimal solution, which corresponds to a hitting set $X$ of $G$ of cost equal to $\OPT(G,c)=\SA_0(G,c)$. This proves \eqref{it:path1}, \eqref{it:path2} when $G$ is a path.
 
 Now, let $G$ be a cycle, and let $v\in V(G)$. Suppose that $v$ belongs to a minimum cost hitting set $X$ of $G$. Then $X\setminus \{v\}$ is a minimum cost hitting set of $G-v$, hence it can be found efficiently since $G-v$ is a path. By iterating over all $v\in V(G)$ and taking the hitting set of minimum cost, we efficiently solve \cVD{} for $(G,c)$. This concludes the proof of \eqref{it:path1}.
 
 Finally, let $\tilde{x}$ be an optimal solution of $\SA_0(G,c)$. First, assume that there is some $v \in V(G)$ such that $\tilde{x}_v \geqslant 1/2$. Since $G - v$ is a path, the optimal hitting set $X'$ in $G-v$ has cost $c(X') \leqslant  \sum_{u \neq v} c(u)\tilde{x}_u$. Hence, we see that $X:= X' + v$ is a hitting set of $H$ with $c(X) = c(v) + c(X') \leqslant c(v) + \sum_{u \neq v}c(u) \tilde{x}_u \leqslant 2\sum_{u} c(u) \tilde{x}_u = 2\SA_0(G,c)$.
On the other hand, if $\tilde{x}_v < 1/2$ for all vertices $v\in V(G)$, then the constraint $\tilde{x}_u + \tilde{x}_v + \tilde{x}_w \geqslant 1$  (where $u, w$ are the neighbors of $v$) implies that there can be no vertex $v$ with $\tilde{x}_v = 0$. So $0 < \tilde{x}_v < 1/2$ for all $v \in V(G)$. 
Therefore, extreme point $\tilde{x}$ is the unique solution of $|G|$ equations of the form $x_u + x_v + x_w = 1$ for $\{u,v,w\} \in \mathcal{P}_3(G)$. Hence $\tilde{x}_v = 1/3$ for all vertices. Thus $\SA_0(G,c) = 1/3 \cdot c(V(G))$. Now notice that since $G$ is a cycle we can partition its vertices into two disjoint hitting sets $X$ and $Y$. Without loss of generality assume that $c(X) \leqslant 1/2 \cdot c(V(G))$. Then $c(X) \leqslant 3/2 \cdot \SA_0(G,c)$. This concludes the proof of \eqref{it:path2}.
 \end{proof}


%



\begin{theorem}\label{thm:SA1CVD_UB}
There is a polynomial-time algorithm that, given a graph $G$ and $c : V(G) \to \R_{\geqslant 0}$, outputs a hitting set $X$ of $G$ such that $c(X) \leqslant 5/2 \cdot \SA_1(G,c)$. In particular, the integrality gap of $\SA_1(G)$ is at most $5/2$. 
\end{theorem}

\begin{proof}
Let $\bar{x}$ be an optimal solution of $\SA_1(G,c)$.
Let $U=\{v\in V(G): \bar{x}_v\geq 2/5\}$, and $H=G\setminus U$. Notice that the restriction of $\bar{x}$ to $V(H)$ is a feasible solution for $\SA_1(H,c)$ whose components are all strictly less than $2/5$. Hence, by Lemma \ref{lem:diagonal}, $H$ cannot contain a $P_3$ which has a diagonal. We will now show that, after possibly getting rid of twins, $H$ is a disjoint union of paths and cycles. Then, by applying Lemma \ref{lem:path-cycles}, we will obtain a minimum cost hitting set of $H$, which, together with $U$, will form our desired hitting set.

First, if $H$ contains twins $u$ and $v$, we can delete $v$, and set $c'(u) := c(u)+c(v)$, $c'(w) := c(w)$ for $w \in V(H) \setminus \{u,v\}$ to obtain a smaller weighted graph $(H',c')$. We claim that $\SA_1(H',c') \leqslant \SA_1(H,c)$.
To see this, we show that one can turn any feasible solution $x$ of $\SA_1(H,c)$ into a feasible solution of $\SA_1(H',c')$ without increasing the cost. Let
$x'_u := \min (x_u,x_v)$ and $x'_w := x_w$ for $w \neq u,v$. It is easy to check that this defines a feasible solution $x'$ to $\SA_1(H',c')$, whose cost is
$$
\sum_{w \neq v} c'(w) x'_w = 
(c(u) + c(v)) \min(x_u,x_v) + \sum_{w \neq u,v} c(w) x_w 
\leqslant \sum_w c(w) x_w\,.
$$
This proves the claim. Moreover, a hitting set $X$ for $H$ can be immediately obtained from a hitting set $X'$ of $H'$ by adding $v$ if and only if $u\in X'$. Finally, notice that there is a feasible solution for $\SA_1(H',c')$ obtained from the restriction of $\bar{x}$ whose components are all strictly less than $2/5$. Hence, from now on we assume that $H$ is twin-free. 

Now, we claim that $H$ is triangle-free and claw-free. Suppose that $H$ contains a triangle with vertices $u$, $v$ and $w$. Since $H$ is twin-free, every edge of the triangle has a distinguisher. Without loss of generality, $\mathcal{P}_3(H)$ contains $\{u,v,y\}$, $\{u,w,y\}$, $\{v,w,z\}$ and $\{u,v,z\}$ where $y, z$ are distinct vertices outside the triangle. It is easy to see that, for instance, edge $uw$ is a diagonal contained in a $P_3$, a contradiction. A similar argument shows that $H$ cannot contain a claw.
This proves the claim, and implies that $H$ has maximum degree at most 2.  That is, $H$ is a disjoint union of paths and cycles. By part \eqref{it:path1} of Lemma \ref{lem:path-cycles}, (applied to each component of $H$), we obtain a minimum cost hitting set $X$ of $H$ which, thanks to \eqref{it:path2} of Lemma \ref{lem:path-cycles}, satisfies $c(X)\leq 2\SA_0(H,c)\leq \frac{5}{2}\SA_1(H,c)$.

Finally, consider $X\cup U$. Clearly, it is a hitting set of $G$. Moreover, one has 
\[c(X\cup U)=c(X)+c(U)\leqslant \frac{5}{2}\left( \SA_1(H,c) + \frac{2}{5}c(U) \right )\leqslant \]
\[
\frac{5}{2}  \left( \SA_1(G,c) -\sum_{u\in U}c(u)\bar{x}_u + \frac{2}{5}c(U) \right ) \leqslant \frac{5}{2} \cdot \SA_1(G,c).
\]

It follows that the integrality gap of $\SA_1(G)$ is at most $5/2$, concluding the proof.
\end{proof}

We now complement the result above by showing a lower bound on the integrality gap of $\SA_1(G,c)$.

\begin{theorem} \label{thm:SA1CVD_LB}
For every $\eps > 0$ there is some instance $(G,c)$ of \cVD{} such that $\OPT(G,c) \geqslant (5/2 - \eps) \SA_1(G,c)$.
\end{theorem}
\begin{proof}
We show there is a graph $G$ for which every hitting set $X$ has $c(X) \geqslant (5/2 - \eps) \SA_1(G,c)$ for $c := \mathbf{1}_G$. Let $G$ be a graph whose girth is at least $k$ for some constant $k \geqslant 5$ and with the independence number $\alpha(G) \leqslant n/k$ where $n := |G|$. It can be shown via the probabilistic method that such a $G$ exists, see \cite{alonspencer2016}. Set $c(v) := 1$ for all $v \in V(G)$. We have $c(X) \geqslant n(1-2/k)$ for every hitting set $X$. To see this observe that since $G$ is triangle-free and $\alpha(G) \leqslant n/k$, when we remove $X$ we will get at most $n/k$ components each of size at most 2. 
Thus there are at most $2n/k$ vertices in $G-X$, so $|X| \geqslant n-2n/k$. Therefore, $\OPT(G,c) \geqslant (1-2/k) n$.

In order to show $\SA_1(G,c) \leqslant 2n/5$, we construct the following feasible solution $x$ to $\SA_1(G,c)$. Set $x_{v} := 2/5$ for all $v \in V(G)$ and $x_{vw} := 0$ if $vw \in E(G)$ and $x_{vw} := 1/5$ if $vw \notin E(G)$. The inequalities defining $\SA_1(G)$ are all satisfied by $x$. This is obvious for inequalities \eqref{ineq:SAtype1}, \eqref{ineq:SAtype3} and \eqref{ineq:SAbounds}. For inequality \eqref{ineq:SAtype2}, notice that at most one of $uz$, $vz$, $wz$ can be an edge of $G$, since otherwise $G$ would have a cycle of length at most $4$. Thus \eqref{ineq:SAtype2} is satisfied too, $x \in \SA_1(G)$ and $\SA_1(G,c) \leqslant 2n/5$.

This completes the proof since, by taking $k \geqslant 5/\eps$, we have $\OPT(G,c) \geqslant n(1-2/k) \geqslant (5/2-\eps) 2n/5 \geqslant (5/2-\eps) \SA_1(G,c)$.
\end{proof}

We now show that the integrality gap decreases to $2+\eps$ after applying $\mathrm{poly}(1/\eps)$ rounds of Sherali-Adams. We first need the following lemma.

\begin{lemma}\label{lem:itrounding}
Fix $\alpha \geqslant 1$ and $r \in \Z_{\geqslant 0}$. Let $(G,c)$ be a minimum order weighted graph such that $\OPT(G,c) > \alpha \cdot \SA_r(G,c)$. The following two assertions hold:
\begin{enumerate}[(i)]
\item if $x$ is an optimal solution to $\SA_r(G,c)$, then $x_v < 1/\alpha$ for all $v \in V(G)$;
\item $G$ is connected and twin-free.
\end{enumerate}
\end{lemma}

\begin{proof}
(i) Suppose for a contradiction that $x_v \geqslant 1/\alpha$, for some $v \in V(G)$. Note that $x$ restricted to $V(G) \setminus \{v\}$ is a feasible solution to $\SA_r(G-v,c)$. Thus $\SA_r(G-v,c) \leqslant \SA_r(G,c) - c(v)x_v$. By the minimality of $G$, there is a hitting set $X'$ of $G-v$ such that $c(X') \leqslant \alpha \cdot \SA_r(G-v,c)$. Therefore $X:= X' + v$ is a hitting set of $G$ with $c(X) = c(v) + c(X') \leqslant c(v) + \alpha \cdot \SA_r(G-v,c) \leqslant \alpha \cdot c(v) x_v + \alpha \cdot \SA_r(G-v,c) \leqslant \alpha \cdot \SA_r(G,c)$, a contradiction.

(ii) Note that $G$ is connected, otherwise there exists a connected component $H$ of $G$ such that $\OPT(H,c_H) > \alpha \cdot \SA_r(H,c_H)$, where $c_H$ is the restriction of $c$ to $V(H)$, contradicting the minimality of $G$.
To show that $G$ is twin-free, we proceed exactly as in the proof of Theorem \ref{thm:SA1CVD_UB}. 
\end{proof}


 \begin{theorem} \label{thm:SA_2+eps}
 For every fixed $\eps > 0$, performing $r = \mathrm{poly}(1/\eps)$ rounds of the Sherali-Adams hierarchy produces an LP relaxation of \cVD{} whose integrality gap is at most $2+\eps$. That is, $\OPT(G,c) \leqslant (2+\eps) \SA_r(G,c)$ for all weighted graphs $(G,c)$.
 \end{theorem}

\begin{proof}
In order to simplify the notation below, let us assume that $2/\eps$ is integer. For instance, we could restrict to $\eps = 2^{-l}$ for some $l \in \Z_{\geqslant 1}$. This does not hurt the generality of the argument. We take $r := 1 + (2/\eps)^4$. We may assume that $\eps < 1/2$ since otherwise we invoke Theorem~\ref{thm:SA1CVD_UB} (taking $r=1$ suffices in this case).

Let $(G,c)$ be a counterexample to the theorem, with $|G|$ minimum. By Lemma ~\ref{lem:itrounding}.(i), for every optimal solution $x$ to $\SA_r(G,c)$, every vertex $v \in V(G)$ has $x_v < 1/(2+\eps)$. By Lemma~\ref{lem:itrounding}.(ii), $G$ is twin-free (and connected). 

We will use the following fact several times in the proof: for all $R \subseteq V(G)$ with $|R| \leqslant r$ and every $x \in \SA_r(G)$, the restriction of $x$ to the variables in $R$ is a convex combination of hitting sets of $G[R]$.  This is easy to see since, denoting by $x_R$ the restriction of $x$, we get $x_R \in \SA_r(G[R])$ and the Sherali-Adams hierarchy is known to converge in at most ``dimension-many'' rounds, see for instance~\cite{conforti2014integer}.

First, we claim that $G$ has no clique of size at least $2/\eps$. Suppose otherwise. Let $C$ be a clique of size $k := 2/\eps$ and let $D$ be a minimal set such that each edge of $C$ has a distinguisher in $D$. Let $H := G[C \cup D]$. Then, following the construction from Section \ref{sec:truetwins}, one can obtain a cost function $c_H$ such that $c_H(H) = 2k-1$, and $c_H(X) \geqslant k-1$ for any hitting set $X$ of $H$. See \cite[Lemma~3]{FJS20} for the full construction, whose proof also shows that $|H| \leqslant 2k-1 \leqslant r$.
Since every valid inequality supported on at most $r$ vertices is valid for $\SA_r(G)$, the inequality $\sum_v c_H(v) x_v \geqslant k-1$ is valid for $\SA_r(G)$. Since $c_H(H) = 2k-1$, this implies that for all $x \in \SA_r(G)$, there is some vertex $a \in V(H)$ with $x_a \geqslant (k-1)/(2k-1)$. Since $(k-1)/(2k-1) \geqslant 1/(2+\eps)$, we get a contradiction. This proves our first claim.

Second, we claim that for every $v_0 \in V(G)$, the subgraph of $G$ induced by the neighborhood $N(v_0)$ has no stable set of size at least $2/\eps$. The proof is similar to that for cliques given above, except that this time we let $H$ be the induced star $K_{1,k}$ with apex $v_0$ and $k = 2/\eps$. The cost function $c_H$ given by Lemma~\ref{lem:hittingclique2good} has $c_H(v_0) = k-1$ and $c_H(v) = 1$ for all $v \in S$. Notice that once again $c_H(H) = 2k-1$. The \emph{star inequality} $\sum_v c_H(v) x_v \geqslant k-1$ is valid for $\SA_r(G)$, which guarantees that for every $x \in \SA_r(G)$ there is some $a \in V(H)$ which has $x_a \geqslant (k-1)/(2k-1) \geqslant 1/(2+\eps)$. This establishes our second claim.

Third, we claim that the neighborhood of every vertex $v_0$ induces a chordal subgraph of $G$. Suppose that $C$ is a hole in $G[N(v_0)]$. We first deal with the case $|C| \leqslant r-1 = (2/\eps)^4$. We can repeat the same proof as above, letting $H$ be the induced wheel on $V(C) + v_0$ and using the cost function $c_H$ defined in the proof of Lemma~\ref{lem:wheel}. Consider the \emph{wheel inequality} $\sum_{v} c_H(v) x_v \geqslant k-3$, where $k := |H| = |C|+1$. Since the wheel has at most $r$ vertices, the wheel inequality is valid for $\SA_r(G)$. Since $c_H(H) = 2k-6 = 2(k-3)$, for every $x \in \SA_r(G)$, there is some $a \in V(H)$ which has $x_a \geqslant 1/2 \geqslant 1/(2+\eps)$. This concludes the case where $|C|$ is ``small''.

Now, assume that $|C| \geqslant r$, and consider the wheel inequality with right-hand side scaled by $2/(2+\eps)$.
Suppose this inequality is valid for $\SA_r(G)$. This still implies that some vertex $a$ of $H$ has $x_a \geqslant 1/(2+\eps)$, for all $x \in \SA_r(G)$, which produces the desired contradiction. It remains to prove that the scaled wheel inequality is valid for $\SA_r(G)$.

  Let $F$ denote any $r$-vertex induced subgraph of $H$ that is a fan.\footnote{A \emph{fan} is a graph obtained from a path by adding an apex vertex.} Hence, $F$ contains $v_0$ as an apex vertex, plus a path on $r-1$ vertices. Letting $c_F(v_0) := r-3 - \lfloor (r-1) / 3 \rfloor$ and $c_F(v) := 1$ for $v \in V(F-v_0)$, we get the inequality $\sum_{v} c_F(v) x_v \geqslant r-3$, which is valid for $\SA_r(G)$. By taking all possible choices for $F$, and averaging the corresponding inequalities, we see that the inequality
\begin{align*}
&\left( r-3 - \left\lfloor \frac{r-1}{3} \right\rfloor \right) x_{v_0} 
+ \frac{r-1}{k-1} \sum_{v \in V(H-v_0)}  x_v \geqslant r-3\\
\iff & \frac{k-1}{r-1} \left( r-3 - \left\lfloor \frac{r-1}{3} \right\rfloor \right) x_{v_0} 
+ \sum_{v \in V(H-v_0)}  x_v \geqslant \frac{k-1}{r-1}(r-3)
\end{align*}
is valid for $\SA_r(G)$. It can be seen that this inequality dominates the scaled wheel inequality, in the sense that each left-hand side coefficient is not larger than the corresponding coefficient in the scaled wheel inequality, while the right-hand side is not smaller than the right-hand side of the scaled wheel inequality. Therefore, the scaled wheel inequality is valid for $\SA_r(G)$. This concludes the proof of our third claim.

By the first, second and third claim\footnote{The first inequality follows since $|H|\leqslant \alpha(H)\cdot\omega(H)$, for every perfect graph $H$.}, $|N(v_0)| \leqslant \omega(G[N(v_0)]) \cdot \alpha(G[N(v_0)]) \leqslant 4/\eps^2$ for all choices of $v_0$. This implies in particular that $|N_{\leqslant 2}[v_0]| \leqslant 1+16/\eps^4 = r$. Now let $H := G[N_{\leqslant 2}[v_0]]$.
Theorem \ref{thm:localratio} applies since $G$ is twin-free, by our second claim. Let $c_H$ be the corresponding cost function. 
The inequality $\sum_{v} c_H(v) x_v \geqslant \OPT(H,c_H)$ is valid for $\SA_r(G)$.

Let $\lambda^*$ be defined as in Step~\ref{step:lambda_star} of Algorithm~\ref{algo}, and let $a \in V(G)$ denote any vertex such that $(c - \lambda^* c_H)(a) = 0$. By minimality of $G$, there exists in $(G',c') := (G-a,c - \lambda^* c_H)$ a minimal hitting set $X'$ of cost $c'(X') \leqslant (2+\eps)\SA_r(G',c')$. We let $X := X'$ in case $X'$ is a hitting set of $G$, and $X := X' + a$ otherwise. Assume that $X = X' + a$, the other case is easier. We have 
\begin{align*}
c(X) &= c'(X') + \lambda^* c_H(X)\\
&\leqslant (2+\eps) \SA_r(G',c') + \lambda^* (c_H(H)-1)\\
&\leqslant (2+\eps) \SA_r(G',c') + 2 \lambda^* \OPT(H,c_H)\\
&\leqslant (2+\eps) \left(\SA_r(G',c') + \lambda^* \OPT(H,c_H)\right)\,.
\end{align*}
By LP duality, we have $\SA_r(G,c) \geqslant \SA_r(G',c') + \lambda^* \OPT(H,c_H)$. This implies that $c(X) \leqslant (2+\eps) \SA_r(G,c)$, contradicting the fact that $(G,c)$ is a counterexample. This concludes the proof.
\end{proof}

We now complement the result above by showing that every  
LP relaxation of \cVD{} with (worst case) integrality gap at most $2-\eps$ must have super-polynomial size. The result is a simple consequence of an analogous result of \cite{bazzi2019no} on the integrality gap of \VC, and of the straightforward reduction from \VC{} to \cVD{}.

\begin{prop}\label{prop:EFlowerbound}
For infinitely many values of $n$, there is a graph $G$ on $n$ vertices such that every size-$n^{o(\log n/ \log \log n)}$ LP relaxation of \cVD{} on $G$ has integrality gap $2 - o(1)$.
\end{prop}
\begin{proof}
In \cite{bazzi2019no} a similar result is proved for LP-relaxations of \VC: for infinitely many values of $n$, there is a graph $G$ on $n$ vertices such that every size-$n^{o(\log n/ \log \log n)}$ LP relaxation of \VC{} on $G$ has integrality gap at least $2 - \eps$, where $\eps=\eps(n)=o(1)$ is a non-negative function. 

Let $G$ be such a graph, and let $G^+$ be the graph obtained from $G$ by attaching a pendant edge to every vertex. It is easy to see that $U \subseteq V(G)$ is a hitting set for $G^+$ if and only if $U$ is a vertex cover of $G$. 

Toward a contradiction, suppose that $Ax \geqslant b$ is a size-$n^{o(\log n / \log \log n)}$ LP relaxation of \cVD{} on $G^+$ with integrality gap at most $2-\delta$, for a fixed $\delta > \eps$ (where $x \in \R^d$ for some dimension $d$ depending on $G$). For every $c^+ \in\Q^{V(G^+)}_{\geqslant 0}$ there exists a hitting set $X$ of $G^+$ such that $c^+(X) \leqslant (2-\delta) \LP(G^+,c^+)$.

We can easily turn $Ax \geqslant b$ into an LP relaxation for \VC. For every vertex cover $U$ of $G$, we let the corresponding point be the point 
$\pi^U \in \R^d$ for $U$ seen as a hitting set in $G^+$. For every $c \in \Q_{\geqslant 0}^{V(G)}$, we define $c^+ \in \Q^{V(G^+)}_{\geqslant 0}$ via $c^+(v) := c(v)$ for $v \in V(G)$, and $c^+(v) := \sum_{u \in V(G)} c(u)$ for $v \in V(G^+) \setminus V(G)$. Then, we let the affine function $f_c$ for $c$ be the affine function $f_{c^+}$ for $c^+$.

Since the integrality gap of $Ax \geqslant b$, seen as an LP relaxation of \cVD, is at most $2 - \delta$, for every $c \in \Q_{\geqslant 0}^{V(G)}$ there exists a hitting set $X$ of $G^+$ whose cost is at most $(2-\delta) \LP(G^+,c^+)$, where $c^+$ is the cost function corresponding to $c$. If $X$ contains any vertex of $V(G^+) \setminus V(G)$, we can replace this vertex by its unique neighbor in $V(G)$, without any  increase in cost. In this way, we can find a vertex cover $U$ of $G$ whose cost satisfies $c(U) \leqslant c^+(X) \leqslant (2-\delta) \LP(G^+,c^+) = (2-\delta)\LP(G,c)$. Hence, the integrality gap of $Ax \geqslant b$ as an LP relaxation of \VC{} is also at most $2-\delta < 2 - \eps$. As the size of $Ax \geqslant b$ is $n^{o(\log n / \log \log n)}$, this provides the desired contradiction.
\end{proof}

We point out that the size bound in the previous result can be improved. Kothari, Meka and Raghavendra~\cite{KMR17} have shown that for every $\eps > 0$ there is a constant $\delta = \delta(\eps) > 0$ such that no LP relaxation of size less than $2^{n^{\delta}}$ has integrality gap less than $2 - \eps$ for Max-CUT. Since Max-CUT acts as the source problem in~\cite{bazzi2019no}, one gets a $2^{n^{\delta}}$ size lower bound for \VC{} in order to achieve integrality gap 
$2 - \eps$. This also follows in a black-box manner from~\cite{KMR17} and~\cite{BPR18}. The proof of Proposition~\ref{prop:EFlowerbound} shows that the same bound applies to \cVD.
\section{Conclusion} \label{sec:conclusion}

In this paper we provide a tight approximation algorithm for the cluster vertex deletion problem (\cVD). Our main contribution is the efficient construction of a local cost function on the vertices at distance at most $2$ from any vertex $v_0$ such that every minimal hitting set of the input graph has local cost at most \emph{twice} the local optimum. If the subgraph induced by $N(v_0)$ (the first neighborhood of $v_0$) contains a hole, or a $2P_3$, then this turns out to be straightforward. The most interesting case arises when the local subgraph $H$ is twin-free, has radius $1$, and moreover $H[N(v_0)] = H - v_0$ is chordal and $2P_3$-free. Such graphs are very structured, which we crucially exploit.

Lemma~\ref{lem:without_v0} allows us to define the local cost function on the vertices distinct from $v_0$ and then later adjust the cost of $v_0$. We point out that condition~\eqref{it:c3} basically says that the local cost function should define a hyperplane that ``almost'' separates the hitting set polytope and the clique polytope of the chordal, $2P_3$-free graph $H - v_0$. This was a key intuition which led us to the proof of Theorem~\ref{thm:localratio}. If these polytopes were disjoint, this would be easy. But actually it is not the case since they have a common vertex (as we show, $H - v_0$ has a hitting clique).

One natural question arising from our approach of \cVD{} in general graphs is the following: is the problem polynomial-time solvable on \emph{chordal} graphs? This seems to be a non-trivial open question, also mentioned in \cite{cao2018vertex}, where similar vertex deletion problems are studied for chordal graphs. It could well be that \cVD{} in general chordal graphs is hard. Now, what about chordal, \emph{$2P_3$-free} graphs? We propose this last question as our first open question.

Our second contribution are polyhedral results for the \cVD~problem, in particular with respect to the tightness of the Sherali-Adams hierarchy. Our results on Sherali-Adams fail to match the 2-approximation factor of our algorithm (by epsilon), and we suspect this is not by chance. We believe that, already for certain classes of triangle-free graphs, the LP relaxation given by a bounded number of rounds of the Sherali-Adams hierarchy has an integrality gap strictly larger than $2$. 
Our intuition goes as follows. Consider the \emph{star inequality} $(k-1) x_{v_0} + \sum_{i=1}^k x_{v_i} \geqslant k-1$, valid when $N(v_0)=\{v_1,\dots, v_k\}$ is a stable set.
Capturing all star inequalities is sufficient to achieve an integrality gap of at most 2 for all triangle-free graphs ~\cite[Algorithm 1]{FJS20}. However, we suspect that Sherali-Adams will have a hard time recovering these in a constant number of rounds. The star inequality is very similar to the clique inequality $\sum_{i=1}^k x_{v_i} \geqslant k-1$, which is valid for \VC{} when $\{v_1,\dots, v_k\}$ is a clique. It is known that Sherali-Adams is unable to capture all clique inequalities in a constant number of rounds of the \VC{} relaxation (see ~\cite[Section 6.1]{laurent03} for an equivalent statement on clique inequalities for the stable set polytope).
Whether this intuition is accurate is our second open question. 

As mentioned already in the introduction, we do not know any polynomial-size LP or SDP relaxation with integrality gap at most $2$ for \cVD. In order to obtain such a relaxation, it suffices to derive each valid inequality implied by Lemmas~\ref{lem:wheel}, \ref{lem:2p3free}, \ref{lem:hittingclique2good} and also somehow simulate Lemma~\ref{lem:get_rid_of_true_twins}. Here, different techniques to construct extended formulations (see for instance \cite{aprile2020extended, aprile2021extended, tiwary2020extension}) could be used. A partial result in this direction is that the star inequality has a bounded-degree sum-of-squares proof. 
This implies that a bounded number of rounds of the Lasserre hierarchy provides an SDP relaxation for \cVD{} with integrality gap at most $2$, whenever the input graph is triangle-free. This should readily generalize to the wheel inequalities of Lemma~\ref{lem:wheel} and of course to the inequality of Lemma~\ref{lem:2p3free} (since the underlying graph has bounded size). However, we do not know how for instance to derive the inequalities of Lemma~\ref{lem:hittingclique2good}.  We leave this for future work as our third open question.

Our fourth open question: what is the best running time for Algorithm~\ref{algo}? We think that it is possible to improve on our $\mathcal{O}(n^4)$ upper bound.

Another intriguing problem is to what extent our methods can be adapted to hitting set problems in other $3$-uniform hypergraphs. We mention an open question due to L{\'{a}}szl{\'{o}} V\'egh~\cite{VeghPC}: for which classes of $3$-uniform hypergraphs and which $\eps > 0$ does the hitting set problem admit a $(3-\eps)$-approximation algorithm?

As mentioned in the introduction, \FVST{} (feedback vertex set in tournaments) is another hitting set problem in a $3$-uniform hypergraph, which is also UGC-hard to approximate to a factor smaller than $2$. There is a recent \emph{randomized} $2$-approximation algorithm~\cite{LMMPPS20}, but no deterministic (polynomial-time) algorithm is known. Let us repeat here the relevant open question from~\cite{LMMPPS20}: does \FVST{} admit a deterministic $2$-approximation algorithm?  

\section*{Acknowledgements}

We are grateful to Daniel Lokshtanov for suggesting Lemma~\ref{lem:chordal+2p3free}, which allowed us to simplify our algorithm and its proof. We also thank two anonymous referees for their helpful comments, which improved the presentation of the paper.

\bibliographystyle{abbrv}
\bibliography{references}


%
%


\end{document}